\documentclass[a4paper,12pt]{amsart}
\usepackage[utf8]{inputenc}
\usepackage[T1]{fontenc}
\usepackage[UKenglish]{babel}
\usepackage[margin=30mm]{geometry}

\usepackage{color}

\usepackage{graphicx}

\usepackage{amsmath,amssymb,amsfonts,amsthm}
\usepackage{mathrsfs,eucal,dsfont}
\usepackage{verbatim,enumitem}

\usepackage{hyperref,url}

\renewcommand{\geq}{\geqslant}

\newcommand{\R}{\mathds R}

\newcommand{\N}{\mathds N}

\newcommand{\Ee}{\mathds E}

\newcommand{\I}{\mathds 1}

\def\aa{\alpha}
\def\dd{\delta}
\def\d{{\rm d}}
\def\<{\langle}
\def\>{\rangle}

\def\LL{\Lambda}
 \def\tt{\tilde}
 \def\ff{\frac}
 \def\ss{\sqrt}
\def\bb{\beta}

\def\R{\mathbb R}  \def\ff{\frac} \def\ss{\sqrt} 
\def\N{\mathbb N} \def\kk{\kappa} 
\def\dd{\delta}   
\def\<{\langle} \def\>{\rangle}  
    
\def\d{\text{\rm{d}}} \def\bb{\beta} \def\aa{\alpha} 
   
 \def\beq{\begin{equation}}  
 
\def\e{\text{\rm{e}}}    
 \def\tt{\tilde}

  \def\ll{\lambda}

  \def\LL{\Lambda}
   
\def\to{\rightarrow}
\def\8{\infty}\def\3{\triangle}
\def\1{\lesssim}

\renewcommand{\tilde}{\widetilde}

\newtheorem{theorem}{Theorem}[section]
\newtheorem{lemma}[theorem]{Lemma}
\newtheorem{proposition}[theorem]{Proposition}

\theoremstyle{definition}

\newtheorem{remark}[theorem]{Remark}

\numberwithin{equation}{section}
\begin{document}
\allowdisplaybreaks

\title[Exponential
ergodicity for two-factor affine processes] {Coupling 
methods and exponential ergodicity for two-factor affine processes}

\author{
Jianhai Bao\qquad\,
\and\qquad
Jian Wang}
\date{}
\thanks{\emph{J.\ Bao:} Center for Applied Mathematics, Tianjin University, 300072  Tianjin, P.R. China. \url{jianhaibao13@gmail.com}}
\thanks{\emph{J.\ Wang:}
College of Mathematics and Informatics \& Fujian Key Laboratory of Mathematical Analysis and Applications (FJKLMAA), Fujian Normal University, 350007 Fuzhou, P.R. China. \url{jianwang@fjnu.edu.cn}}

\maketitle

\begin{abstract}
In this paper, by invoking the
 coupling approach, we establish
 exponential ergodicity under the $L^1 $-Wasserstein distance for two-factor affine processes. The method employed herein is universal in a certain sense
so that it is applicable to
general two-factor affine processes, which allow that the first component solves a general CIR process,
and that there are interactions in the second component,
as well as that the Brownian noises are correlated;
and even to some models beyond two-factor processes.

\medskip


\end{abstract}

\section{Introduction}\label{section1}
An affine process on the state space $D:=\R_+^m\times\R^n$, where
$m,n\in \N_0:=\{0,1,2,\cdots\}$ with $m+n\ge1$,   is a
time-homogeneous Markov process which
satisfies that the logarithm
of characteristic function for the transition distribution of such a
process is affine with respect to the initial state $x\in D$; see
the pioneer work  \cite{DPS}  upon  the general
  theory and \cite{DFS} on succinct mathematical foundations and  complete characterizations of regular
  affine processes. Nowadays, the theory of affine processes has
  been developed in various directions; see e.g.
  \cite{CPGU,FMS,KM,KRST1,KRST2}. Meanwhile,  affine processes have been
  applied   considerably in mathematical finance due to their computational
  tractability and flexibility in capturing many  empirical
  features of financial series; see the book
\cite{A} and references within.

The set of affine processes contains a large class of important Markov processes such as continuous-state branching processes with
 immigration with the state space $D=\R^m_+$, and Ornstein-Uhlenbeck type processes with the state space $D=\R^n$. The long time behaviors
 (for example,  existence and uniqueness
 of the stationary probability measure, ergodicity,  exponential ergodicity, and so on) of those two
  processes have been studied extensively in the literature;
 see e.g.
  \cite{LM,Masuda,SY, W1}. In order to study long time
  behaviors for general affine processes on the canonical state space $D=\R_+^m\times\R^n$, it is natural to start from the following simplest (but interesting and important) two-factor affine process on $\R_+\times \R$ (that is, $m=n=1$):
\begin{equation}\label{WW0}
\begin{cases}
\d Y_t=(a-b Y_t)\,\d t+  Y_{t-}^{{1}/{\bb}} \,\d L_t,&\quad t\ge0,\,Y_0\ge0,
\\
\d X_t=(\kappa-\lambda X_t)\,\d t+ Y_{t-}^{1/\aa}\,\d Z_t,&\quad t\ge0,\,X_0\in \R,
\end{cases}
\end{equation}
where $a\ge 0$, $b,\kappa,\lambda\in
\R,$ $\beta, \aa\in (1,2]$, $(L_t)_{t\ge0}$ is a
  spectrally  positive $\bb$-stable process with the L\'{e}vy measure
$\nu_\bb(\d z):=C_\bb z^{-(1+\bb)}\I_{\{z>0\}}\d z$ with $C_\bb=(\bb\Gamma(-\bb))^{-1}$ (where $\Gamma$ denotes the Gamma
 function) in case of
  $\bb\in (1,2)$, a standard Brownian motion (which will be denoted by $(B_t)_{t\ge0}$ later) in case of
    $\bb=2$, and
  similarly $(Z_t)_{t\ge0}$ is a spectrally  positive $\aa$-stable process in case of
  $\aa\in (1,2)$,
  a standard Brownian motion in case of
   $\aa=2$. We further assume that $(L_t)_{t\ge0}$ and $(Z_t)_{t\ge0}$ are  mutually independent.

There are a few of
results on this direction, where $(Z_t)_{t \ge0}$ in  \eqref{WW0} is a standard Brownian motion (i.e.,
$\aa=2$). The existence of a unique stationary distribution was
addressed in \cite{BDLP}, and furthermore the exponential ergodicity
for the case that $(L_t)_{t\ge0}$ is a standard Brownian motion
(i.e., $\aa=2$) was also investigated therein.  Subsequently, the
corresponding results derived in \cite{BDLP} were extended to a much more
general setup in \cite{BP} (see (1.1) therein for more details).
Later, \cite{JKR} complemented the result on the exponential
ergodicity in \cite{BDLP} by allowing $(L_t)_{t\ge0}$ to be any
spectrally  positive $\bb$-stable process with $\bb\in(1,2).$ As
mentioned in \cite{BDLP}, once the existence of a unique stationary
distribution and the exponential ergodicity of the two-factor
affine process
solved by the SDE \eqref{WW0} are available, the asymptotic analysis
of estimators for parameter estimations (e.g.\ least squares
estimation and maximum-likelihood estimation)  via certain
(discrete-time or continuous-time) observations \cite{BDLP14,BP19}
can be implemented.

However, ergodic properties  of the SDE \eqref{WW0} for the case that $(Z_t)_{t\ge0}$ is a spectrally  positive $\aa$-stable process with
$\aa\in(1,2)$ are still open. We will fill the gap in this paper. To taste our contribution, we herein  state the result for the following SDE  on $\R_+\times \R$:
\begin{equation}\label{WW1}
\begin{cases}
\d Y_t=(a-b Y_t)\,\d t+  Y_{t}^{{1}/{2}} \,\d B_t,&\quad t\ge0,\,Y_0=y\ge0,\\
\d X_t=(\kappa-\lambda X_t)\,\d t+ Y_{t-}^{1/\aa}\,\d Z_t,&\quad t\ge0,\,X_0=x\in \R,
\end{cases}
\end{equation}where $a\ge 0$, $b,\kappa,\lambda\in
\R,$ $(B_t)_{t\ge0}$ is a standard Brownian motion, and $(Z_t)_{t\ge0}$ is an independent
spectrally  positive $\aa$-stable process with $\aa\in (1,2]$.
 In terms of
\cite[Theorem 2.1]{BDLP}, the SDE \eqref{WW1} has a unique strong solution $(Y_t,X_t)_{t\ge0}$.
Let $P(t, (y,x),\cdot)$ be the transition probability kernel of the process $(Y_t,X_t)_{t\ge0}$ with the initial value $(y,x)$.

For a strictly increasing  function $\psi$ on $\R^2_+$ and two
probability measures $\mu_1$ and $\mu_2$ on $\R_+\times \R$, define
$$ W_{\psi}(\mu_1,\mu_2)=\inf_{\Pi\in\mathscr{C}(\mu_1,
\mu_2)}\int_{\R_+^2\times \R^2} \psi(|y-\tt y|, |x-\tt x|)\,\Pi(\d
y,\d \tt y, \d x , \d \tt x), $$ where $\mathscr{C}(\mu_1, \mu_2)$
is the collection of all probability
measures on $\R_+^2\times \R^2$
with marginals $\mu_1$ and $\mu_2$. When $\psi$ is concave, the
above definition yields a Wasserstein distance $W_\psi$ in
the space of probability measures $\mu$ on $\R_+\times \R$ such  that
$\int_{\R_+\times \R}\psi(y,|x|)\,\mu(\d y,\d x)<\infty$. If
$\psi(s,t)=(s^2+t^2)^{1/2}$ for all $s,t\geq0$, then $W_\psi$ is the
standard $L^1$-Wasserstein distance, which will be  denoted simply by $W_1$.
Another well known example for $W_\psi$ is given by
$\psi(s,t)=\I_{\R_+^2\setminus\{(0,0)\}}(s,t)$, which leads to the
total variation distance $$W_\psi(\mu_1,
\mu_2)=\frac{1}{2}\|\mu_1-\mu_2\|_{\rm Var}:=\frac{1}{2}
[(\mu_1-\mu_2)^+(\R_+\times \R)+(\mu_1-\mu_2)^-(\R_+\times \R)]$$
where $(\mu_1-\mu_2)^+$ and $(\mu_1-\mu_2)^-$ stand respectively for the positive part and the negative part for the Jordan-Hahn decomposition of the signed measure $\mu_1-\mu_2.$

\begin{theorem}\label{T:main} Let $(Y_t, X_t)_{t\ge0}$ be the unique strong solution to \eqref{WW1}. If $b>0$ and $\lambda>0$, then the process $(Y_t,X_t)_{t\ge0}$ is exponentially ergodic with respect to the $L^1$-Wasserstein distance $W_1$, i.e., there is a constant $\eta_0>0$ such that for all $y\ge0$, $x\in \R$ and $t>0$,
\begin{equation}\label{YY}
W_{1} (P(t, (y,x),\cdot), \mu)\le C_0(y,x){\rm e}^{-\eta_0 t},
\end{equation}
where $C_0(y,x)>0$ is independent of $t$.
\end{theorem}

We emphasize that, to investigate ergodicity of the two-factor
affine processes, the Foster-Lyapunov criteria  
\cite{MT} developed 
by Meyn-Tweedie
 was adopted in \cite{BDLP,BP,JKR},
where the
key ingredient is to examine the irreducibility of skeleton chains.
Whereas, in general, it is a cumbersome task to check the
irreducible property. Instead of the approach above, in this work we shall
take advantage of the probabilistic coupling argument. Compared with
\cite{BDLP,BP,JKR}, the coupling method enjoys several advantages.
For instance, it avoids verification of irreducible property of
skeleton chains;
on the other hand,  it is universal in a certain
sense  that all the frameworks in \cite{BDLP,BP,JKR} can be handled.

The coupling technique has been applied to study the exponential ergodicity  under both the $L^1$-Wasserstein distance and the total variation norm
in \cite{Lwcw, Maj}  for SDEs
driven by L\'evy noises, and in \cite{LW}
for general continuous-state nonlinear branching processes, which in particular include continuous-state branching processes (that is,
typical class of
affine processes on $\R_+$).
 In contrast to
 these quoted papers, there are essential differences to realize
the coupling approach for two-factor processes we are concerned with  in the present paper.
Roughly speaking, we do not consider  directly the coupling of two-factor processes. Whereas we first couple
the first component. For the two-factor process solved by \eqref{WW1}, the marginals of the coupling process for the first component will stay together once they meet at the first time.
 Also due to the structure of \eqref{WW1}, we then  deal with the coupling concerning the second component as well. From the point of view above,
  the constructions of 
  proper coupling processes  and appropriate Lyapunov functions as well as their refined  estimates, which are crucial to adopt the coupling approach,
   require
    much more effort than that in \cite{LW, 
    Lwcw, Maj}. See Remarks \ref{E:ggg} and \ref{R-P} for more details.

Recently there are a few of developments on the topics related with ergodicty of affine processes. For example, see \cite{JKR1} for the existence of limit distributions for affine processes, \cite{FJR} for  exponential ergodicity in Wasserstein
distances for affine processes, and \cite{FJKR1,MS} and the references therein for exponential/geometric ergodicity of affine processes on cones, and so on.
In particular,
the exponential ergodicity of affine processes in terms of
suitably chosen Wasserstein distances has
been established in \cite[Theorem 1.5]{FJR} under
the first moment condition on the state-dependent and log-moment conditions on the state-independent jump measures, respectively.
Similar ideas have been used in \cite{FJKR} to study the exponential ergodicity for SDEs of nonnegative processes with jumps.
Here are two main differences between the approach of \cite[Theorem 1.5]{FJR} and  the counterpart 
in our paper.
\begin{itemize}
\item[(i)] Applying \cite[Theorem 1.5(a)]{FJR} to the setting of Theorem \ref{T:main}, one may get the exponential ergodicity of the two-factor process  defined by \eqref{WW1}  in terms of the $L^1$-Wasserstein distance $W_1$. The intermediate key step to yield \cite[Theorem 1.5(a)]{FJR} is \cite[Proposition 7.3]{FJR}, which claims that there exist constants $K$ and $\delta>0$ such that for all $y,\tt y\in \R_+$, $x,\tt x\in \R$ and $t>0$,
    $$
        W_{\psi^*} (P(t,(y,x),\cdot), P(t,(\tt y,\tt x),\cdot))\le K \e^{-\delta t}  \psi^*(|y-\tt y|,|x-\tt x|),$$ where
$$\psi^*(u,v):=(u+ u^{1/2}) +v,\quad u,v\ge0;$$ see \cite[(7.2)]{FJR}. In the present paper, we indeed can verify that for any $\theta\in(0,1)$, there are positive constants $\eta:=\eta(\theta)$ and $C:=C(\theta)$ such that for all $y,\tt y\in \R_+$, $x,\tt x\in \R$ and $t>0$,
$$
W_{\psi_\theta}   (P(t,(y,x),\cdot), P(t,(\tt y,\tt x),\cdot))\le C \e^{-\eta t} \psi_\theta|y-\tt y|,|x-\tt x|),
$$
where
$$\psi_\theta(u,v):=(u+ u^\theta)+ v,\quad u,v\ge0;$$ see \eqref{e:rrr} below.

\item[(ii)] The proof of \cite{FJR} is based on
the characterization of affine processes, see e.g. the proof of the crucial statement \cite[Proposition 6.1]{FJR}.
 However, our approach does not rely heavily on
 the structure of two-factor processes.
 Indeed, our argument still works for
 some models beyond two-factor processes; see Subsection \ref{subsection4.3}.
\end{itemize}
\ \

The remainder of this paper is arranged as follows. In Section \ref{sec2}, we overview the existing result on existence and uniqueness of non-negative solutions to \eqref{WW1}, reveal that the first order moment of solutions to \eqref{WW1} is finite,
and construct the coupling operator by applying  the coupling by reflection for a small distance and the synchronous
coupling for a big distance to  the
first component and the synchronous
coupling to  the second component in \eqref{WW1}. The Section \ref{Section3} is devoted to the proof of Theorem \ref{T:main} via the coupling approach and by constructing appropriate Lyapunov functions. In Section \ref{Section4}, we aim to apply the ideas adopted in
 Sections \ref{sec2} and \ref{Section3} to
 general two-factor affine models, (which allow that the first component solves a general CIR process, that there are interactions in the second component,
and even that the Brownian noises are correlated), as well as some models beyond two-factor models.

\section{Coupling for two-factor processes}\label{sec2}

\subsection{Preliminary:
existence and uniqueness of strong solutions
}\label{section2.1}
In this part, we recall some known results on
existence and uniqueness of strong solutions to the SDE \eqref{WW1}, and we also investigate moment estimates for the corresponding solution.

The existence and uniqueness of non-explosive strong solutions
 to the SDE \eqref{WW1}
follows
 essentially from \cite[Theorem 2.1]{BDLP}. Roughly speaking, the pathwise uniqueness of non-negative strong solution
 $(Y_t)_{t\ge0}$ of the first equation in \eqref{WW1} is due to  the well-known Yamada-Watanabe approximation approach. Once the first component
$(Y_t)_{t\ge0}$ is available, the second component $(X_t)_{t\ge0}$ solved  by the second equation in \eqref{WW1} is indeed  a one-dimensional Ornstein-Uhlenbeck type process.

According to the It\^{o} formula, the infinitesimal generator of the process $(Y_t,X_t)_{t\ge0}$ associated with the SDE
\eqref{WW1} is given by
\begin{equation}\label{e:gene}
\begin{split} (L f)(y,x)= & (a-by) \partial_1 f(y,x)+\frac{y}{2} \partial_{11} f(y,x)+(\kappa-\lambda x) \partial_2 f(y,x)\\
&+y\int_{\R}\left(f(y,x+z)-f(y,x)-\partial_2 f(y,x)z\right)\,\nu_\alpha(\d z)
\end{split}
\end{equation} for any $f\in C^2_b(\R_+\times \R)$. Here and below, $\partial_if$ stands for the first order
derivative w.r.t.\ the $i$-th component, and $\partial_{ij}f$ means
the second order derivative w.r.t.\ the $i$-th component followed by
the $j$-th component.

Now, we take $W(y,x)=1+y+h(x)$, where $0\le h\in C^2(\R)$ such that $h(x)=|x|$ for all $|x|\ge2$ and $\|h'\|_\infty+\|h''\|_\infty<\infty.$ Then, according to \eqref{e:gene}, it follows that
\begin{align*}(LW)(y,x)= & a-by
+(\kappa -\lambda x)h'(x)+y\int_{\R}\left(h(x+z)-h(x)-h'(x)z\right)\,\nu_\alpha(\d z)\\
\le & a+|b| y+ \|h'\|_\infty (|\kappa|+|\lambda|\cdot |x|)\\
&+\frac{1}{2}y\|h''\|_\infty \int_{\{|z|\le 1\}}z^2\,\nu_\alpha(\d z)+ 2 y\|h'\|_\infty \int_{\{|z|\ge1\}}|z|\,\nu_\alpha(\d z)\\
\le & C_0(1+y+h(x))=C_0W(y,x),
 \end{align*} where $C_0>0$ is independent of $y$ and $x$. Hence, for any $(y,x)\in \R_+\times \R$ and $t>0$,
 $$\Ee^{(y,x)} W(Y_t,X_t)\le W(y,x)+C_0\int_0^t \Ee^{(y,x)} W(Y_s,X_s)\,\d s .$$
  This, along with Gronwall's inequality, yields
   that for any $(y,x)\in \R_+\times \R$ and $t>0$,
 $$\Ee^{(y,x)} W(Y_t,X_t)\le W(y,x) \e^{C_0t}.$$ In particular, the first order moment of $(Y_t,X_t)_{t\ge0}$ is finite,
 i.e., for any $(y,x)\in \R_+\times \R$ and $t>0$,
\begin{equation}\label{e:mom}\Ee^{(y,x)}( Y_t+|X_t|)<\infty.
\end{equation}

\subsection{Markovian coupling for two-factor affine processes}\label{L:order}
Now, we consider the following SDE on $\R_+\times \R$:
\begin{equation}\label{Cou1}
\begin{cases}
\d \tt Y_t=\begin{cases}(a-b \tt Y_t)\,\d t+(-\I_{\{0<|Y_t-\tt Y_t|<1\}}+\I_{\{Y_t-\tt Y_t|\ge1\} })\tt Y_{t}^{1/2}\d B_t,&\,\, t< T_Y,\\
(a-b \tt Y_t)\,\d t+\tt Y_{t}^{1/2}\d B_t,&\,\,t\ge T_Y,\end{cases}\\
\d \tt X_t=(\kappa-\lambda \tt X_t)\,\d t+ {\tt
Y_{t-}}^{1/\alpha}\,\d L_t,
\end{cases}
\end{equation} with the initial value $(\tt Y_0,\tt X_0)=(\tt y,\tt x)\in \R_+\times\R $, where $${T_Y}:=\inf\{t\ge0: Y_t=\tt Y_t\}.$$
Define $$B_t^*=\begin{cases}B_0+\displaystyle \int_0^t \Big(-\I_{\{0<|Y_s-\tt Y_s|<1\}}+\I_{\{Y_s-\tt Y_s|\ge1\} }\Big)\,\d  B_s,&\quad 0\le t\le T_Y,\\
B_{T_Y}^*+B_t-B_{T_Y},&\quad t\ge T_Y.\end{cases}$$ Then,
$$\d\tt Y_t=(a-b \tt Y_t)\,\d t+ \tt Y_{t}^{1/2}\d B_t^*.$$  Since $(B_t^*)_{t\ge0}$ is
still a standard Brownian motion and
the SDE \eqref{WW1} has a
unique strong solution, the SDE \eqref{Cou1} also admits a unique
strong solution $(\tt Y_t,\tt X_t)_{t\ge0}$ so that $Y_t=\tt Y_t$ for all $t\ge T_Y$;
moreover, the process $(\tt Y_t, \tt X_t)_{t\ge0}$ enjoys the same
law (i.e.\  the
same
transition
probabilities) as that of
$(Y_t,X_t)_{t\ge0}$. In particular, $((Y_t,X_t),(\tt Y_t,\tt
X_t))_{t\ge0}$ is a non-explosive coupling of the process $(Y_t,X_t)_{t\ge0}$.
Roughly speaking, before two marginal processes meet we will apply the coupling by reflection for the
first component $(Y_t)_{t\ge0}$ when the distance of them is less than $1$  and the synchronous
coupling when the distance of them is large or equal to $1$, and
once two marginal processes meet we will adopt the synchronous
coupling; while we always take the synchronous
coupling for the second component $(X_t)_{t\ge0}$.
We remark
that, due to the continuity of $t\mapsto Y_t$ and $t\mapsto \tt Y_t$ almost surely and $Y_t=\tt Y_t$ for all $t\ge T_Y$, the coupling of the first component process preserves the order
property; that is, $Y_t\ge \tt Y_t$ for any $t\ge0$ when $Y_0\ge \tt
Y_0.$

Furthermore, it is not hard to see that the generator (later we call it the coupling operator of the generator $L$ given by
\eqref{e:gene}) of the coupling process  $((Y_t,X_t),(\tt Y_t,\tt X_t))_{t\ge0}$ is given by
\begin{equation}\label{L2}
\begin{split}&(L^*f)(y,\tt y, x, \tt x)\\
&=(a-b y)\partial_1f(y,\tt y, x,\tt x)+\frac{y}{2}\partial_{11} f(y,\tt y, x,\tt x)\\
&\quad+(a-b \tt y)\partial_2f(y,\tt y, x,\tt x)+\frac{\tt y}{2}
\partial_{22}f(y,\tt y, x,\tt x)\\
&\quad-\sqrt{y
\tt y}\partial_{12}f(y,\tt y, x,\tt x)\I_{\{0<y-\tt y<1\}}+\sqrt{y
\tt y}\partial_{12}f(y,\tt y, x,\tt x)\I_{\{  y-\tt y\ge1\}\cup \{y=\tt y\}}\\
&\quad+(\kappa-\lambda x)\partial_3f(y,\tt y, x,\tt x)+(\kappa-\lambda \tt x)\partial_4f(y,\tt y, x,\tt x)\\
&\quad +\tt y \int_0^\infty  (f(y,\tt y, x+z,\tt x+z)-f(y,\tt y,x,\tt x)\\
&\quad~~~~~~~~~~~~-\partial_3f (y,\tt y,x,\tt x)z-\partial_4f (y,\tt y,x,\tt x)z )\,\nu_\aa(\d z)\\
&\quad +(y-\tt y)\int_0^\infty\left(f(y,\tt y, x+z,\tt x)-f(y,\tt
y,x,\tt x)-\partial_3f (y,\tt y,x,\tt x)z\right)\,\nu_\aa(\d z)
\end{split}
\end{equation} for any $f\in C^2(\R_+^2
\times \R^2)$ and $(y,\tt y, x,\tt x)\in \R_+^2\times \R^2$ with
$y\ge\tt y$.
Similarly, we can write the expression of $(L^*f)(y,
\tt y, x, \tt x)$ for $f\in C^2(\R_+^2
\times \R^2)$ and $(y,\tt y, x,\tt x)\in \R_+^2\times \R^2$ with $y
\le \tt y$.

\section{Exponential convergence in the Wasserstein distance
}\label{Section3}
This section is devoted to the proof of Theorem \ref{T:main}.
Throughout this section, we shall fix
$\aa\in
(1,2)$, which is the stability index  for the spectrally positive
$\aa$-stable process $(Z_t)_{t\ge0}$ in the second equation of
\eqref{WW1}. Below,  let $\nu_\aa(\d z)$ be the associated L\'evy measure of $(Z_t)_{t\ge0}$.

We begin with the following simple lemma.

\begin{lemma}\label{L:lem1} Let
   $g\in
C_b^2(\R_+)$ such that $0\le g\le 1$. Then, for the function
\begin{equation}\label{EE3}
F(s,t):=\left(1-g\left({t}/{s}\right)\right)s +g\left({t}/{s}\right)t,\quad s,t>0,
\end{equation}
it holds that
\begin{equation}\label{E6}
 |\partial _iF(s,t)|\le c_0,\quad
|\partial_{ii}F(s,t)|\le c_0 s^{-1},\qquad 1 \le
{t}/{s}\le 2,~i=1,2,
\end{equation}
where $c_0>0$ is independent of $s,t$, and  $\partial _iF$
$($resp. $\partial_{ii}F$$)$ stands for the first $($resp.
second$\, )$ order derivative w.r.t. the $i$-th component of the
function $F.$
\end{lemma}
\begin{proof}
A straightforward calculation shows that for any $s,t>0$,
\begin{align*}
\partial_1F
(s,t)&=g'(t/s) {t}{s^{-1}}
+ (1-g(t/s)) -g'({ t}/{s}){t^2}{s^{-2}}, \\
\partial_{11}F
(s,t)&=-g''(t/s){t}^2{s^{-3}}+
2g'({ t}/{s}){t^2}{s^{-3}}+g''({t}/{s}){t^3}{s^{-4}}, \\
\partial_2F
(s,t)&=-g'(t/s) +g'(t/s)t{s}^{-1}+
g(t/s),\\
\partial_{22}F
(s,t)&=-g''(t/s)s^{-1}+g''(t/s)t{s^{-2}}+2
g'(t/s) {s}^{-1}.
\end{align*}
Then,  \eqref{E6} follows by $g\in C_b^2(\R_+)$.
\end{proof}

Next, we will take $g\in C_b^2(\R_+)$ with  $g'\ge0$ such that
\begin{equation}\label{e:test}
g(r)=
\begin{cases}
0,&0\le r<1,\\
(r-1)^{2+\delta},&1<r<{3}/{2},\\
1,&r\ge2
\end{cases}\end{equation} for some constant $\delta>0$. With this choice,
$F(s,t)
=s$ if $t\le s$; $F(s,t)
=t$ if $t\ge2s.$

Now, for any $c>0$ and $\theta\in(0,1)$ we define the function
\begin{equation}\label{v:function}V_{c,\theta}(s, t)=c (s+s^\theta)+ F (s, t),~~s,t\ge0.\end{equation}
 It is clear from $g\in[0,1]$ that for any $c>0$,
\begin{equation}\label{e:testaa}V_{c,\theta}(s,t)\asymp (s\vee s^\theta) +  t,~~s,t\ge0.\end{equation} Herein, we use the shorthand
notation $f\asymp g$ for two non-negative functions $f$ and $g$, which means that there exists a constant $c\ge1$ such that $c^{-1}f\le
g\le c\, f$ on the domain.

Below, concerning $V_{c,\theta}^*(y,\tt y,  x,\tt x):=V_{c,\theta}(y-\tt y, |x-\tt x|)$ for any $(y,\tt y, x,\tt x)
\in \R_+^2\times \R^2$ with $y\ge \tt y$,
we will  simply
write $(L^* V_{c,\theta})(y-\tt y, |x-\tt x|):=  (L^*  V_{c,\theta}^*)(y,\tt y,  x,\tt x)$.
We now have the following statement, which is crucial for the proof of Theorem \ref{T:main}.
\begin{proposition}\label{P:P-1}
For any $\theta\in (0,1)$, there exist constants $c,\zeta>0$ such that for any $y>\tilde y\ge0$
 and $x,\tilde x\in \R$,
\begin{equation}\label{R4}
(L^*V_{c,\theta})(y-\tilde y,|x-\tilde x|)\le -\zeta V_{c,\theta} (y-\tilde y,|x-\tilde x|).
\end{equation}
 \end{proposition}

\begin{proof} (1) Let $G
\in C^2(\R_+^2)$. According to the definition of the coupling
operator $L^*$, we find that for any $y\ge \tilde y\ge0$ and $x,\tilde
x\in \R$,
\begin{equation}\label{E5}
\begin{split}
&(L^*G)(y-\tt y,|x-\tt x|)\\
&=-b\partial_1G(y-\tt y,|x-\tt x|)(y-\tt y)\\
&\quad +\ff{1}{2}(\ss
y+\ss{\tt y})
^2 \partial_{11}G(y-\tt y,|x-\tt x|)\I_{\{0<y-\tt y<1\}}\\
&\quad+\ff{1}{2}(\ss
y-\ss{\tt y})
^2 \partial_{11}G(y-\tt y,|x-\tt x|)\I_{\{  y-\tt y\ge1\}}\\
&\quad -\lambda \partial_2G (y-\tt y,|x-\tt x|)|x-\tt x|\\
&\quad+(y -\tt y )\int_0^\infty\bigg(G(y -\tt y ,|x+z-\tt
x|)-G(y-\tt
y,|x-\tt x|)\\
&\quad~~~~~~~~~~~~~~~~~~~~~~ -\partial_2 G (y-\tt y,|x-\tt
x|)\ff{(x-\tt x)z}{|x-\tt x|}\bigg)\,\nu_\aa(\d z),
\end{split}
\end{equation} where we set $\ff{(x-\tt x)z}{|x-\tt x|}=|z|$ if $x =\tt x$, and $\partial_iG$ $($resp.\ $\partial_{ii}G $$)$ stands for the first $($resp.
second$\, )$ order derivative w.r.t. the $i$-th component of the
function $G.$

Below,   let  $y>\tt y\ge0 $ and $x,\tt x\in\R$ be arbitrary. For any $c>0$ and $\theta\in(0,1),$
let  $$U_{c,\theta}(s,t)=c\, (s+  s^\theta),~~~s,t\ge0.
$$
 Then,
 we get from  \eqref{E5} and $ \theta\in(0,1)$  that \begin{equation*}
 \begin{split}
(L^*U_{c,\theta})(y-\tt y,|x-\tt x|)&=-c\, b(1+\theta (y-\tt y)^{\theta-1})(y-\tt
y)\\
&\quad-\ff{1}{2}c\,\theta(1-\theta)(\ss
y+\ss{\tt y})
^2  (y-\tt y)^{\theta-2}\I_{\{0<y-\tt y<1\}}\\
&\quad-\ff{1}{2}c\,\theta(1-\theta)(\ss
y-\ss{\tt y})
^2 (y-\tt y)^{\theta-2}\I_{\{ y-\tt y\ge1\}}\\
&\le -c\, b(y-\tt
y)-c\, b\,\theta (y-\tt y)^\theta\\
&\quad -\ff{1}{2}c\,\theta(1-\theta)(\ss
y+\ss{\tt y})
^2  (y-\tt y)^{\theta-2}\I_{\{0< y-\tt y<1\}}.
\end{split}
\end{equation*}
Thanks to \eqref{v:function}, this yields \eqref{R4}
provided that there is a constant $C>0$ such that
\begin{equation}\label{Y0}
\begin{split}
(L^*F)
(y-\tt y,|x-\tt x|)&\le -\lambda   |x-\tilde
x|  +  2( \ll+ C) (y-\tt y)\\
&\quad+   C(\ss y+\ss{\tt
y})^2(y-\tt y)^{ -1}\I_{\{0<y-\tt y<1\}}
\end{split}
\end{equation}
and  by taking
 \begin{equation*}
c =\ff{4(\ll+C)}{b}+\ff{4C}{\theta(1-\theta)}.
\end{equation*}

Whereas, \eqref{Y0} is available as long as there exists a constant $C >0$ such that
\begin{equation}\label{PP}
\begin{split}
(L^*F)
(y-\tt y,|x-\tt x|)&\le -\lambda |x-\tt
x| \I_{\{|x-\tt x|\ge 2(y-\tt y) \}} +C(y-\tt y)+ \ff{C }{2}    (y-\tt y)^{2-\aa} \\
&\quad +
\ff{C }{2}  (\ss y+\ss{\tt
y})^2(y-\tt y)^{-1}\I_{\{0< y-\tt y<1\}}\\
&\quad+\ff{C }{2} (\ss y-\ss{\tt
y})^2(y-\tt y)^{-1}\I_{\{ y-\tt y\ge1\}}.
\end{split}
\end{equation}
Indeed, \eqref{Y0} holds true since \eqref{PP} implies
\begin{align*}
(L^*F)(y-\tt y,|x-\tt x|)
&\le  -\lambda   |x-\tilde
x|  + 2( \ll+C) (y-\tt y)\\
&\quad+   C(\ss y+\ss{\tt
y})^2(y-\tt y)^{ -1}\I_{\{0<y-\tt y<1\}},
\end{align*}
 where  we used the facts that
\begin{equation}\label{YYY}
(y-\tt y)^{2-\aa}\le (y-\tt y)_{\{y-\tt y\ge1\}} +(\ss y+\ss{\tt
y})^2(y-\tt y)^{ -1}\I_{\{0<y-\tt y<1\}}
\end{equation}
and, for $y-\tt y\ge1$,
\begin{equation*}
(\ss y-\ss{\tt
y})^2(y-\tt y)^{-1} =(\ss y-\ss{\tt
y})/(\ss y+\ss{\tt
y})\le 1\le y-\tt y.
\end{equation*}
So, to achieve the desired assertion \eqref{R4}, it is sufficient to show \eqref{PP} for
 the following three cases:
\begin{enumerate}
\item[{\bf (i)}]$|x-\tt x|\ge 2(y-\tt y) $;
\item[{\bf (ii)}]$|x-\tt x|\le y-\tt y$;
\item[{\bf (iii)}]$ y-\tt y \le |x-\tt x| \le 2(y-\tt y)$.
\end{enumerate}

 \medskip

\noindent {\bf   Proof of  \eqref{PP}   for the case  (i)}.
  In this case,
$F(y-\tt y,|x-\tt x|)= |x-\tt x|$  so that, by
\eqref{E5}, one has
$$
(L^*F) (y-\tt y,|x-\tt x|) =-\lambda |x-\tt
x|+(y-\tt y)(I_1+I_2+I_3),$$ where
\begin{equation}\label{L3}
\begin{split}I_1&:=\int_{\{|x+z-\tt x|< y-\tt y \}}\left(y-\tt
y - |x-\tt x| -  \ff{(x-\tt x)z}{|x-\tt
x|}\right)\,\nu_\aa(\d z),\\
I_2&:= \int_{\{|x+z-\tt x|>2(y-\tt y)\}}\left(   |x+z-\tt x| -
|x-\tt x| -   \ff{ (x-\tt x)z}{|x-\tt
x|}\right)\,\nu_\aa(\d z),\\
I_3&:= \int_{\{y-\tt y\le |x+z-\tt x|\le 2(y-\tt
y)\}}\Bigg[\left(1-g\left(\ff{|x+z-\tt x|}{y-\tt y}\right)\right) (y-\tt
y)  \\
&\qquad + g\left(\ff{|x+z-\tt x|}{y-\tt y}\right)  |x+z-\tt x| -
|x-\tt x| -  \ff{(x-\tt x)z}{|x-\tt
x|}\Bigg]\,\nu_\aa(\d
z).\end{split}\end{equation}

Note that, if $|x+z-\tt x|< y-\tt y $ and $|x-\tt x|\ge2(y-\tt y)$,
then $$ |z|\ge |x-\tt x|-|x+z-\tt x|>|x-\tt x|-(y-\tt y)> y-\tt y.$$
Whence, there exists a constant $c_1>0$ such that
\begin{equation}\label{W1}
 I_1 \le  \int_{\{z>y-\tt y\}}\big(  y-\tt y +
 z \big)\,\nu_\aa(\d z)\\
 \le   c_1 (y-\tt
y)^{1 -\aa}.
\end{equation}

A simple calculation shows that
\begin{equation}\label{W2}
\begin{split}
I_2
&=\int_{\{|x+z-\tt x|>2(y-\tt y)\}}\left(  (x+z-\tt x) -
(x-\tt x) -   z\right) \I_{\{x\ge \tt x\}}\,\nu_\aa(\d z)\\
&\quad + \int_{\{|x+z-\tt x|>2(y-\tt y)\}}\left(  -(x+z-\tt x) +
(x-\tt x) +   z\right) \I_{\{x< \tt x,x-\tt x+z\le 0\}}\,\nu_\aa(\d z)\\
&\quad + 2 \int_{ \{ z >\tt x-x+2(y-\tt y)\}}
( x-\tt x+z )  \I_{\{x<\tt x,x-\tt x+z>0\}}\,\nu_\aa(\d z)\\
&\le2\int_{\{  z\ge4(y-\tt y)\}}
 z    \I_{\{x<\tt x,x-\tt x+z>0\}}\,\nu_\aa(\d z)\le c_2   ( y-  \tt y )^{1-\aa}
\end{split}
\end{equation}
for some constant $c_2>0$,
where in the
 inequality we used $x<\tt x$  and
 $\tt x-x\ge2(y-\tt y)$.

Observe that
\begin{equation}\label{W3}
\begin{split}
I_3
&= \int_{\{y-\tt y\le |x+z-\tt x|\le 2(y-\tt y)\}}
\!\!\!\left(1\!-\!g\left(\ff{|x+z-\tt x|}{y-\tt y}\right)\!\right)(  y\!-\tt y\! -
|x+z-\tt
x|)\,\nu_\aa(\d z)\\
&\quad+\int_{\{y-\tt y\le |x+z-\tt x|\le 2(y-\tt y)\}}\left(   |x+z-\tt x| -
|x-\tt x| -   \ff{ (x-\tt x)z}{|x-\tt
x|}\right)\,\nu_\aa(\d z)\\
&\le2 \int_{\{y-\tt y\le x+z-\tt x\le 2(y-\tt y)\}}
(x-\tt x+z)  \I_{\{x<\tt x,x+z-\tt x>0\}}\,\nu_\aa(\d z)\\
&\le 2 \int_{\{ z\ge 3(y-\tt y)\}}
 z  \,\nu_\aa(\d z)= c_3(y-\tt y)^{1-\aa}
\end{split}
\end{equation}
for some constant $c_3>0,$
where in the first inequality we used $g\in[0,1]$ and $y-\tt y\le |x+z-\tt x|$, and
  the second inequality follows from  $x<\tt x$ and $\tt x-x\ge2(y-\tt y)$.
So, combining \eqref{W1} with   \eqref{W2},
\eqref{W3}
and \eqref{YYY} yields \eqref{PP} for the case {\bf (i)}.

\noindent{\bf Proof of \eqref{PP} for the case (ii)}.  In this case, $F
(y-\tt y,|x-\tt x|)= y-\tt y $. Then, according to
\eqref{E5}, we have
\begin{align*}
&(L^*F )
(y-\tt y,|x-\tt x|)\\
&= -   b(y-\tt
y)\\
&\quad+ (y-\tt y)\bigg\{\int_{\{|x-\tt x+z|>2(y-\tt y)\}}\left(|x-\tt
x+z|- (y-\tt
y)\right)\nu_\aa(\d z)\\
&\qquad\qquad\quad\quad  + \int_{\{y-\tt y<|x-\tt x+z|<2(y-\tt y)\}} g\left(\ff{|x-\tt
x+z|}{ y-\tt y }\right) ( |x-\tt x+z|-(y-\tt
y)) \, \nu_\aa(\d z)\bigg\} \\
&=:-  b(y-\tt y)+(y-\tt y)(J_1+J_2).
\end{align*}

It is easy to  get that
$$
J_1
 \le \int_{\{|x+z-\tt x|>2(y-\tt y)\}} ( |x -\tt
x| +z )
 \nu_\aa(\d z) \le   2\int_{\{z > y-\tt y \}}  z
 \nu_\aa(\d z) =c_1(y-\tt y)^{1-\aa}
$$
for some  $c_1>0,$ where in the second inequality we used
$y-\tt y>|x-\tt x|$ and
$$z>|x-\tt x+z|-|x-\tt x|>2(y-\tt y)-|x-\tt x|\ge y-\tt y.
$$

According to the definition of the function $g(\cdot)$ given by \eqref{e:test}, without loss of generality, we can assume that there is a constant $c_*>0$ such that  $g(r)\le c_* (r-1)^{2+\dd}$ for $r\ge1$. Then, it holds
\begin{align*}
J_2&\le \ff{c_*}{(y-\tt y)^{2+\dd}}\int_{\{y-\tt y<|x-\tt x+z|<2(y-\tt
y)\}} ( |x-\tt x+z|-(y-\tt y)  )^{3+\dd}\,\nu_\aa(\d z)\\
&\le\ff{c_*}{(y-\tt y)^{2+\dd}}\int_{\{y-\tt y<|x-\tt x+z|<2(y-\tt y)\}}
(|x-\tt x|+z-(y-\tt y)
)^{3+\dd}\,\nu_\aa(\d z)\\
&\le \ff{c_* }{(y-\tt y)^{2+\dd}}\int_{\{z<3(y-\tt y)\}}
z^{3+\dd} \, \nu_\aa(\d z)
=c_2(y-\tt y)^{1-\aa}
\end{align*}
for some constant $c_2>0,$ where
the last inequality follows from
the facts that $|x-\tt x|<y-\tt y $ and
 $z <|x-\tt x+z|+|x-\tt x|\le
 3(y-\tt y).$ Henceforth, combining the   estimates above and taking \eqref{YYY} into account
 give \eqref{PP} for the case {\bf (ii)}.

\smallskip

\noindent{\bf Proof of \eqref{PP} for the case (iii)}. With
regard to this case, we derive from \eqref{E6} and \eqref{E5} that there exists a   constant
$c_1>0$ so that
\begin{equation}\label{E0}
\begin{split}
(L^*F)
(y-\tt y,|x-\tt x|)&\le
c_1  (y-\tt y)+(y -\tt y )( \LL_1+\LL_2)\\
&\quad+
 \ff{c_0}{2}(\ss
y+\ss{\tt y})
^2  (y-\tt y)^{-1}\I_{\{0< y-\tt y<1\}}\\
&\quad+\ff{c_0}{2}(\ss
y-\ss{\tt y})
^2 (y-\tt y)^{-1}\I_{\{  y-\tt y\ge1\}},
\end{split}
\end{equation}
where
\begin{align*}\LL_1:=\int_{\{z\le
({y-\tt y})/{2}\}}\bigg(F
(y -\tt y ,|x+z-\tt x|)&-F
(y-\tt y,|x-\tt x|)\\
&  -\partial_2F
(y-\tt y,|x-\tt x|)\ff{(x-\tt x)z}{|x-\tt x|}
\bigg)\,\nu_\aa(\d z)\end{align*} and
\begin{align*}
\LL_2:= \int_{\{z\ge ({y-\tt y})/{2}\}}\bigg(F
(y -\tt y
,|x+z-\tt
x|)&-F
(y-\tt y,|x-\tt x|)\\
& -\partial_2F
(y-\tt y,|x-\tt x|)\ff{(x-\tt x)z}{|x-\tt x|}
\bigg)\,\nu_\aa(\d z).
\end{align*}

Note that  $x+z-\tt x<0$ in case of $x<\tt x$,     $z\le (y-\tt y)/2$ and $y-\tt y\le |x-\tt x|\le 2(y-\tt
y)$. Then, we have
\begin{align*}\LL_1&=\int_{z\le
({y-\tt y})/{2}}\bigg(F
(y -\tt y ,x-\tt x+z)-F
(y-\tt y,x-\tt x)
  -\partial_2F
  (y-\tt y,x-\tt x)z \bigg)\\
&\qquad\qquad\qquad\qquad\qquad\qquad\qquad\qquad\qquad\qquad\qquad\qquad\qquad\times\I_{\{x\ge\tt x\}}\,\nu_\aa(\d z)\\
&\quad+ \int_{\{z\le ({y-\tt y})/{2},x+z-\tt x<0\}} (F
(y
-\tt y ,
\tt x-x-z )-F
(y-\tt y,\tt x-  x )\\
&\qquad\qquad\qquad\qquad\qquad\qquad\qquad\qquad\qquad\quad\,\,\,\,\,
  +\partial_2F
  (y-\tt y,\tt x-  x )z \bigg)\I_{\{x<\tt x\}}\,\nu_\aa(\d z).
\end{align*}
Applying the mean value theorem and taking  \eqref{E6}  into account,
we find that there is a constant $c_2>0$ so that
\begin{align*}
\LL_1 &=\ff{1}{2}\int_{\{z\le ({y-\tt y})/{2}\}}
\partial_{22}F
(y-\tt
y,\xi_1)z^2 \I_{\{x\ge\tt x\}}\nu_\aa(\d z)\\
&\quad+\ff{1}{2}\int_{\{z\le ({y-\tt y})/{2},x+z-\tt x<0\}}
\partial_{22}F
(y-\tt
y,\xi_2 )z^2  \I_{\{x<\tt x\}}\nu_\aa(\d z) \\
&\le c_2(y-\tt y)^{1-\aa},
\end{align*}
where $\xi_1\in[y-\tt y,5(y-\tt y)/2]$ and $\xi_2\in[(y-\tt y)/2,
2(y-\tt y)].$  On the other  hand, it follows from \eqref{E6} that
there exist constants $c_3,c_4,c_5>0$ such that
\begin{align*}
\LL_2&\le c_3\int_{\{z\ge ({y-\tt y})/{2}\}}  \left(y-\tt
y + |x-\tt
x+z| + z\right)\nu_\aa(\d z)\\
&\le c_4\int_{\{z\ge ({y-\tt y})/{2}\}}  \left( y-\tt y
 +z \right)\nu_\aa(\d z)
 \le c_5(y-\tt y)^{1-\aa},
\end{align*}
where  in the second inequality we used the fact that $y-\tt y\le |x-\tt
x|\le 2(y-\tt y)$. Therefore, \eqref{PP} for the case {\bf (iii)} follows by substituting the two estimates
above into \eqref{E0} and taking advantage of \eqref{YYY}.
\end{proof}

\begin{remark}\label{E:ggg}  The main task of the proof
above is to control an upper bound for $(L^*F)
(y-\tt y,|x-\tt x|)$ for all $y>\tt y\ge 0$ and $x,\tt x\in \R$. In cases {\bf (i)} and {\bf (ii)}, we can get that
$$(L^*F)
(y-\tt y,|x-\tt x|)\le  -\lambda |x-\tt
x| \I_{\{|x-\tt x|\ge 2(y-\tt y) \}} +C_0(y-\tt y)^{2-\alpha}.$$ To get \eqref{R4} for  those two cases, one can take
$$V_{c,c_*}(s, t)=c\Big(s+\int_0^s   \e^{{-c_*u^{1-\alpha}}} \d u\Big) + F (s, t),~~~s,t\ge0$$ instead of $V_{c,\theta}(s, t)$ defined by \eqref{v:function} (with possibly choices of $c, c_*>0$), and apply the coupling by reflection for the
first component $(Y_t)_{t\ge0}$ before two marginal processes meet (that is, $$ -\sqrt{y
\tt y}\partial_{12}f(y,\tt y, x,\tt x)\I_{\{0<y-\tt y<1\}}+\sqrt{y
\tt y}\partial_{12}f(y,\tt y, x,\tt x)\I_{\{  y-\tt y\ge1\}\cup \{y=\tt y\}}$$ is replaced by $$-\sqrt{y
\tt y}\partial_{12}f(y,\tt y, x,\tt x)\I_{\{y\neq \tt y\}}+ \sqrt{y
\tt y}\partial_{12}f(y,\tt y, x,\tt x)\I_{\{y=\tt y\}}$$ in the coupling operator $(L^*f)(y,\tt y, x,\tt x)$ given by \eqref{L2}). Note that, the nice property of the function $V_{c,c_*}(s, t)$ above is that
$$V_{c,c_*}(s, t)\asymp s+t,$$ which is comparable to the \lq\lq cost function\rq\rq\,\, in the standard
$L^1$-Wasserstein distance. However, for the case {\bf(iii)}, to eliminate
the last two positive terms in the right hand side of \eqref{E0}, we need not only to apply the test function $V_{c,\theta}(s, t)$ defined by \eqref{v:function}, but also to modify the coupling of  the Brownian motion $(B_t)_{t\ge0}$ in the
first component $(Y_t)_{t\ge0}$.
 \end{remark}

Now, we are in a position to present the
\begin{proof}[Proof of Theorem $\ref{T:main}$]
Note that, for all $\theta\in (0,1)$, $(y,\tt y)\mapsto |y-\tt y|+|y-\tt y|^\theta$ is a metric on $\R_+$, and so $(\mu_1,\mu_2)\mapsto W_{\psi_\theta}(\mu_1,\mu_2)$ is a metric on the space of probability measures on $\R_+\times \R$,
where $$\psi_\theta(u,v):=u+u^{\theta} +v,\quad u,v\ge0.$$ Note that $W_1(\mu_1,\mu_2)\le  W_{\psi_\theta}(\mu_1,\mu_2)$ for any $\theta\in (0,1)$.  To prove Theorem \ref{T:main}, we actually verify
the exponential ergodicity of the process $(Y_t,X_t)_{t\ge0}$ in terms of $W_{\psi_\theta}$.
Furthermore, as mentioned in Subsection \ref{L:order}, the coupling of the first component process preserves the order
property, i.e.,\ $Y_t\ge \tt Y_t$ for all $ t\ge0$ if $Y_0=y>\tt Y_0=\tt y$.
By carrying out more or less
standard arguments (see e.g. the proof
of \cite[Corollary 1.8]{Maj}), to achieve the desired
assertion \eqref{YY} on  the exponential ergodicity of the process $(Y_t,X_t)_{t\ge0}$ in terms of $W_{\psi_\theta}$, it is sufficient to show that
for all $y\ge \tt y\ge0$, $x,\tt x\in \R$ and $t>0$,
\begin{equation}\label{e:rrr}
W_{\psi_\theta} (P(t,(y,x),\cdot), P(t,(\tt y,\tt x),\cdot))\le C \e^{-\eta  t}\psi_\theta(y-\tt y,|x-\tt x|)
\end{equation}
holds with
some constants $C,\eta>0$
(which are independent of $y,\tt y,x,\tt x$ and $t$).
This obviously holds true provided that
\begin{equation}\label{Y1}
\Ee^{(y,x)} \psi_\theta(Y_t,|X_t|)<\infty
\end{equation}
and that
\begin{equation}\label{Y2}
\begin{split}
 \Ee^{((y,x),(\tt y, \tt x))} \psi_\theta(Y_t-\tt Y_t,|X_t-\tt X_t|) \le C\e^{-\eta  t} \psi_\theta(y-\tt y,|x-\tt x|).
\end{split}
\end{equation}

Since $\psi_\theta(u,v)=u+u^\theta +v\le 2(1+u+v)$, \eqref{Y1} is true due to \eqref{e:mom}. Let $V_{c,\theta}(s,t)$ be the function and $\zeta$  the positive constant in Proposition \ref{P:P-1}.
 By using \eqref{e:testaa}, \eqref{Y2}  is valid once we claim that
\begin{equation}\label{Y3}
\Ee^{((y,x), (\tt y,\tt x))} V_{c,\theta}(Y_{t}-\tt Y_{t},|X_{t}-\tt X_{t}|) \le \e^{-(\lambda\wedge \zeta) t}V_{c,\theta}(y-\tt y,|x-\tt x|)
\end{equation}
for all $y\ge \tt y\ge0$, $x,\tt x\in \R$ and $t>0$, where $\lambda>0$ is the constant in the SDE \eqref{WW1}. So, in what follows, it remains to show that the assertion \eqref{Y3} holds.

Define $$T_Y=\inf\{t>0: Y_t=\tt Y_t\},~~
 T_{Y,n}=\inf\{t>0: Y_t-\tt Y_t\le 1/n\},~~~n\ge1.$$
Let $y\ge \tt y\ge 0$. Note that for all $r\ge 0$,
\begin{equation}\label{Y4}V_{c,\theta}(0,r)= r=r\partial_2 V_{c,\theta}(0,r).\end{equation}
If $y=\tt y\ge0$, then $T_Y=0$ so that
for any $t\ge s\ge 0$,
\begin{equation}\label{Y6}
\begin{split}
&\e^{(\lambda\wedge \zeta) t}V_{c,\theta}(Y_{t}-\tt Y_{t},|X_{t}-\tt X_{t}|)\\
&=\e^{(\lambda\wedge \zeta) t}V_{c,\theta}(0,|X_{t}-\tt X_{t}|)\\
&=\e^{(\lambda\wedge \zeta) s}V_{c,\theta}(0,|X_s-\tt X_s|)\\
&\quad+ \int_s^t\!\!\e^{(\lambda\wedge \zeta) r}\{(\lambda\wedge \zeta)V_{c,\theta}(0,|X_r\!-\!\tt X_r|)\!-\! \lambda|X_r\!-\!\tt X_r|\partial_2V_{c,\theta}(0,|X_r\!-\!\tt X_r|)\}\,\d r\\
&\le \e^{(\lambda\wedge \zeta) s}V_{c,\theta}(0,|X_s-\tt X_s|)),
\end{split}
\end{equation}
where the second identity holds true from the structure of the second equation in \eqref{WW1},
i.e., for any fixed $s\ge0$ with $Y_s=\tt Y_s$, we have $Y_r=\tt Y_r$ for all $r\ge s$ and so
\begin{equation}\label{Y6---}\d (X_r-\tt X_r)= - \lambda  (X_r-\tt X_r) \,\d r,\quad r\ge s,\end{equation}
 and
the inequality is owing to \eqref{Y4}. With \eqref{Y6} for $s=0$ at hand, it is easy to see  that \eqref{Y3} holds for $y=\tt y\ge0.$

In the following, we only need  to verify  \eqref{Y3} for   $y>\tt y\ge0$.
 Choose $n_0\ge1$ sufficiently large such that $y-\tt y> 1/n_0$. Noting
 again
 that $Y_t\ge \tt Y_t$ for all $ t\ge0$ (in particular, for all $0\le t\le T_{Y,n}$ and $n\ge 1$) whenever $Y_0=y>\tt Y_0=\tt y$. Then,
for any $t>0$ and $n\ge n_0$, by It\^o's formula, it follows that
\begin{align*}&\Ee^{((y,x), (\tt y,\tt x))}(\e^{\zeta (t\wedge T_{Y,n})} V_{c,\theta}(Y_{t\wedge T_{Y,n}}-\tt Y_{t\wedge T_{Y,n}},|X_{t\wedge T_{Y,n}}-\tt X_{t\wedge T_{Y,n}}| ))\\
&=V_{c,\theta}(y-\tt y, |x-\tt x|)\\
&\quad +\Ee^{((y,x), (\tt y,\tt x))}\int_0^{t\wedge T_{Y,n}}\e^{\zeta s}\left( \zeta  V_{c,\theta}(Y_{s}-\tt Y_{s},|X_{s}-\tt X_{s}| )- L^*V_{c,\theta}(Y_{s}-\tt Y_{s},|X_{s}-\tt X_{s}| )\right) \d s\\
&\le V_{c,\theta}(y-\tt y, |x-\tt x|), \end{align*} where we utilized \eqref{R4} in the last inequality.
Approaching $n\to \infty$ yields  that all $y>\tt y\ge0$, $x,\tt x\in \R$ and $t>0$,
$$ \Ee^{((y,x), (\tt y,\tt x))}(\e^{\zeta (t\wedge T_Y)}V_{c,\theta}(Y_{t\wedge T_Y}-\tt Y_{t\wedge T_Y},|X_{t\wedge T_Y}-\tt X_{t\wedge T_Y}| ))\le V_{c,\theta}(y-\tt y,|x-\tt x|).$$ Consequently,
  \eqref{Y3} is available  for   $y>\tt y\ge0$ by
taking advantage of
\begin{align*}
\e^{(\lambda\wedge \zeta) t}V_{c,\theta}(Y_{t}-\tt Y_{t},|X_{t}-\tt X_{t}|)
&=\e^{(\lambda\wedge \zeta) t}V_{c,\theta}(Y_{t}-\tt Y_{t},|X_{t}-\tt X_{t}|)\I_{\{T_Y>t\}}\\
&\quad+ \e^{(\lambda\wedge \zeta) t}V_{c,\theta}(Y_{t}-\tt Y_{t},|X_{t}-\tt X_{t}|)\I_{\{T_Y\le t\}}\\
&=\e^{(\lambda\wedge \zeta) (t\wedge T_Y)}V_{c,\theta}(Y_{t\wedge T_Y}-\tt Y_{t\wedge T_Y},|X_{t\wedge T_Y}-\tt X_{t\wedge T_Y}|) \I_{\{T_Y>t\}}\\
&\quad+\e^{(\lambda\wedge \zeta) t}V_{c,\theta}(0,|X_t-\tt X_t|)\I_{\{T_Y\le t\}}\\
&\le \e^{(\lambda\wedge \zeta) (t\wedge T_Y)}V_{c,\theta}(Y_{t\wedge T_Y}-\tt Y_{t\wedge T_Y},|X_{t\wedge T_Y}-\tt X_{t\wedge T_Y}|) \I_{\{T_Y>t\}}\\
&\quad+\e^{(\lambda\wedge \zeta) T_Y}V_{c,\theta}(0,|X_{T_Y}-\tt X_{T_Y}|)\I_{\{T_Y\le t\}}\\
&\le \e^{(\lambda\wedge \zeta) (t\wedge T_Y)}V_{c,\theta}(Y_{t\wedge T_Y}-\tt Y_{t\wedge T_Y},|X_{t\wedge T_Y}-\tt X_{t\wedge T_Y}|) \I_{\{T_Y>t\}}\\
&\quad+\e^{(\lambda\wedge \zeta)  (t\wedge T_Y) }V_{c,\theta}(Y_{t\wedge T_Y}-\tt Y_{t\wedge T_Y},|X_{t\wedge T_Y}-\tt X_{t\wedge T_Y}|)\I_{\{T_Y\le t\}}\\
&=\e^{(\lambda\wedge \zeta) (t\wedge T_Y)}V_{c,\theta}(Y_{t\wedge T_Y}-\tt Y_{t\wedge T_Y},|X_{t\wedge T_Y}-\tt X_{t\wedge T_Y}|),
\end{align*}
where in the
first  inequality above we used \eqref{Y6} with $s=T_Y$.
\end{proof}

\begin{remark}\label{R-P} Here we make some comments on the proof of Theorem \ref{T:main}.
\begin{itemize}
\item[(i)]
Recall that $T_Y=\inf\{t>0: Y_t=\tt Y_t\}.$
Define
$$T_X=\inf\{t\ge0: X_t=\tt X_t\},~~~T=\inf\{t>0: Y_t=\tt Y_t, X_t=\tt X_t\}.$$ It is clear that $T\ge T_Y\vee T_X$ and $Y_t=\tt Y_t$ for $t\ge T_Y$. However, by the structure of two-factor process defined by \eqref{WW1}, it can take place that $T>T_Y\vee T_X$, since
 $X_t=\tt X_t$ is not true for all $t\ge T_X$ unless $t\ge T_Y$.
The idea for the proof of Theorem \ref{T:main} is to make full use of the coupling time $T_Y$ for the first component, rather than $T$. To consider the coupling process until the coupling time $T_Y$, we need the crucial estimate \eqref{R4}. Since before time $T_Y$, it may occur that $X_t=\tt X_t$ for some $t\le T_Y$. Hence, if we will apply the It\^{o} formula for the test function $V_{c,\theta}(y-\tt y, |x-\tt x|)$ as given in Proposition \ref{P:P-1}, then this function is required to be differentiable on $\{(x,x):x\in \R\}$ for any fixed $y,\tt y$.
Furthermore, from $T_Y$ to $T$ the coupling of the first component always stays together, so we only need to couple the second component. For this, we make use  of  \eqref{Y6---}, which is due to the special characterization of the two-factor process.

\item[(ii)] As mentioned above, Proposition \ref{P:P-1} is crucial for the proof of Theorem \ref{T:main}. Instead of \eqref{E5}, one natural way to obtain the exponential ergodicity in the $L^1$-Wasserstein distance is to prove that there exists a constant $c_0>0$ such that for all $y>\tt y\ge0$ and $x,\tt x\in \R$,
    \begin{equation}\label{e:remark}(L^*G)(y-\tt y, |x-\tt x|)\le -c_0 G(y-\tt y, |x-\tt x|),\end{equation} where    \begin{equation}\label{e:remark1}G(s,t)\asymp s+t,\quad s,t\ge0;\end{equation} see \cite{LW, 
    Lwcw, Maj} for example. Since in the setting of Theorem \ref{T:main} the drift term satisfies the so-called monotone condition due to $b>0$ and $\lambda>0$, one may just take
    $$G(s,t)=G_0(s,t):=s+t,\quad s,t\ge0.$$ By some
     calculations, we find that there is a constant $c_1> 1$ so that for all $y>\tt y\ge0$ and $\tt x\ge x$,
   \begin{align*}-b (y-\tt y)-\lambda(\tt x-x)+&c_1^{-1}(y-\tt y)(\tt x-x)^{1-\alpha}\\
   &\le (L^* G_0)(y-\tt y, \tt x-x)\\
   &\le -b (y-\tt y)-\lambda(\tt x-x)+c_1(y-\tt y)(\tt x-x)^{1-\alpha}.\end{align*} That
   is, with this choice, \eqref{e:remark} can not be true. Hence,  some modification of $G_0(s,t)$ is required. The function $F(s,t)$ defined by \eqref{EE3} and satisfying \eqref{e:remark1} is one possible candidate.  Whereas, the function $F(s,t)$ above is not enough, we need further to refine it into $V_{c,\theta}(s,t)$.
   In particular, the factor $s^\theta$ is added to balance some bad estimates from $(L^*F)(y-\tt y, |x-\tt x|).$ See Remark \ref{E:ggg} above for some details.   \end{itemize} \end{remark}

\section{Exponential ergodicity for other two-factor affine processes and beyond}\label{Section4}

\subsection{The case \eqref{WW1} with two spectrally positive stable noises}
In this subsection, we are still interested in  the two-factor model
\eqref{WW1} but with  the Brownian motion $(B_t)_{t\ge0}$ in the first equation replaced by  a
  spectrally positive $\bb$-stable process $(L_t)_{t\ge0}$ for some $\bb\in (1,2)$. More precisely,    we shall work on  the SDE on $\R_+\times \R$:
\begin{equation}\label{WW2}
\begin{cases}
\d Y_t=(a-b Y_t)\,\d t+ Y_{t-}^{1/\bb} \,\d L_t,&\quad t\ge0,\,Y_0\ge0,\\
\d X_t=(\kappa-\lambda X_t)\d t+Y_{t-}^{1/\alpha}\,\d Z_t,&\quad t\ge0,\,X_0\in \R,
\end{cases}
\end{equation}
where $a\ge0$, $b,\kappa,\lambda\in\R$,
$(L_t)_{t\ge0}$ (resp.\ $(Z_t)_{t\ge0}$) is a
  spectrally  positive  $\bb$-stable  (resp.\ $\aa$-stable)  process with the L\'{e}vy measure
$\nu_\bb(\d z)$ (resp.\ $\nu_\aa(\d z)$). We further assume that $(L_t)_{t\ge0}$ and
$(Z_t)_{t\ge0}$ are mutually independent.
Again, by means of \cite[Theorem 2.1]{BDLP},  \eqref{WW2} has a unique strong solution $(Y_t, X_t)_{t\ge0}$.
Then, we have the following statement for the affine process associated with the SDE \eqref{WW2}.

\begin{theorem}\label{T:main1} Let $(Y_t, X_t)_{t\ge0}$ be the unique strong solution to \eqref{WW2}. If $b>0$ and $\lambda>0$, then the process $(Y_t,X_t)_{t\ge0}$ is exponentially ergodic with respect to the $L^1$-Wasserstein distance $W_1$.
\end{theorem}

For the SDE \eqref{WW2}, we shall apply the synchronous coupling for both
components $(Y_t)_{t\ge0}$ and $(X_t)_{t\ge0}$. For this, we consider the SDE on $\R_+\times \R$:
$$
\begin{cases}
\d \tt Y_t=(a-b \tt Y_t)\,\d t+ \tt Y_{t-}^{1/\bb} \,\d L_t,&\quad t\ge0,\,\tt Y_0>0,\\
\d \tt X_t=(\kappa-\lambda \tt X_t)\d t+\tt Y_{t-}^{1/\alpha}\,\d
Z_t,&\quad t\ge0,\,X_0\in \R.
\end{cases}
$$
It is clear that $((Y_t,X_t),(\tt Y_t,\tt X_t))_{t\ge 0}$ is a coupling of the process $(Y_t,X_t)_{t\ge0}$.
According to \cite[Corollary 2.3]{LW}, the coupling of the first component process also preserves the order property; that is, $Y_t\ge \tt Y_t
$ for all $t\ge0,$ in case of  $Y_0\ge \tt Y_0.$

Furthermore, the infinitesimal generator (i.e.\ the coupling operator) of the coupling process $((Y_t,X_t),(\tt Y_t,\tt X_t))_{t\ge 0}$ is given by
\begin{equation}\label{WW3}
\begin{split}
&(L^*G)(y-\tt y,|x-\tt x|)\\
&=-b\partial_1G(y-\tt y,|x-\tt x|)(y-\tt
y) -\lambda \partial_2G(y-\tt y,|x-\tt x|)|x-\tt x|\\
&\quad +(y -\tt y )\int_0^\infty\Big(G(y -\tt y+z,|x-\tt x|)-G(y-\tt
y,|x-\tt x|) \\
&\qquad\qquad\qquad\qquad\qquad\qquad\qquad\qquad -\partial_1G(y-\tt y,|x-\tt x|)z\Big)\,\nu_\bb(\d z)\\
&\quad+(y -\tt y )\int_0^\infty\Big(G(y -\tt y,|x+z-\tt x|)-G(y-\tt
y,|x-\tt x|)\\
&\qquad\qquad\qquad\qquad\qquad\qquad\qquad\qquad -\partial_2 G
(y-\tt y,|x-\tt x|)\ff{(x-\tt x)z}{|x-\tt x|}\Big)\,\nu_\aa(\d z)
\end{split}
\end{equation}for   $G\in C^2(\R_+^2)$ and $(y,\tt y, x, \tt x)\in \R_+^2\times \R^2$ with $0\le \tt y\le y$. Similarly, we can write the expression of $(L^*G)(y-\tt y,|x-\tt x|)$ for   $G\in C^2(\R_+^2)$ and $(y,\tt y, x, \tt x)\in \R_+^2\times \R^2$ with $0\le y\le \tt y$.

For  any $c>0$ and $\theta\in(0,1)$, define
$$V_{c,\theta}(s,t) =c\,(s+ s^\theta)+F(s,t),$$
where $F$ was introduced in \eqref{EE3}. It is obvious that
for any $c>0$,
$$V_{c,\theta}(s,t)\asymp (s+s^\theta  )+t.$$

The proof of Theorem \ref{T:main1} is based on the
 following proposition.

\begin{proposition}\label{P:es2} For any $\theta\in(0,2-(\aa\vee \bb)]$, there exist  constants $c,\eta>0$ such that
\begin{equation}\label{T1}
(L^*V_{c,\theta})(y-\tt y,|x-\tt x|)\le -\eta V_{c,\theta} (y-\tt y,|x-\tt x|),~~\tt y\in[0,y),x,\tt x\in\R.
\end{equation}  \end{proposition}
\begin{proof}
Some non-trivial modifications  of the proof for Proposition \ref{P:es2} are required although the corresponding idea is
similar to that of Proposition \ref{P:P-1}. In the following,  we  highlight some key differences. For  $c>0$ and $\theta\in(0,2-(\aa\vee \bb)]$,  set
$$ U_{c,\theta}(s,t):=c\,(s+s^\theta),\quad s,t\ge0.$$
Below, we set $0\le\tt y< y$ and let $x,\tt x\in
 \R$. According to  \eqref{WW3},   we deduce
\begin{equation}\label{T2}
\begin{split}
&(L^*U_{c,\theta})(y-\tt y,|x-\tt x|)\\
&=-b\, c\, \{y-\tt y+\theta(y-\tt y)^\theta\}\\
&\quad+c\,(y -\tt y ) \int_0^\infty\left((y -\tt y+z)^\theta-(y-\tt
y)^\theta-\theta(y-\tt y)^{\theta-1}z\,\right)\,\nu_\bb(\d z)\\
&\le -b c \{y-\tt y+\theta(y-\tt y)^\theta\},
\end{split}
\end{equation}
where in the inequality we used
the fact  that for all $0\le \tt y<y$ and $z\ge0$,
$$(y -\tt y+z)^\theta-(y-\tt
y)^\theta-\theta (y-\tt y)^{\theta-1}z\le 0,\quad \theta\in(0,1)$$  thanks to
 the mean value theorem.
 Thus, \eqref{T1} follows provided that there exists a constant $C_1>0$ such that
\begin{align}\label{K1}
(L^*F)(y-\tt y,|x-\tt x|)
\le   -\lambda  |x-\tilde x|+ C_1\left( y-\tt y  +(y-\tt y)^\theta \right),
\end{align}
 by taking \eqref{T2} into account  and  choosing $c>0$ such that $cb \theta >C_1$. Nevertheless,  to derive \eqref{K1}, it suffices   to show that
there exists a constant $C_2>0$ such that
\begin{equation}\label{K2}
\begin{split}
(L^*F_\theta)(y-\tt y,|x-\tt x|)\le & -\lambda   |x-\tilde x| \I_{\{|x-\tt x|\ge 2(y-\tt y)\}}\\
&+ C_2\left(  y-\tt y  +(y-\tt y)^{2-\aa}+ (y-\tt y)^{2-\bb} \right).
\end{split}
\end{equation}
In fact, \eqref{K2} yields \eqref{K1} by noting that there exists a constant $C_3>0$ such that
\begin{align*}
&-\lambda  |x-\tilde x|\I_{\{|x-\tt x|\ge 2(y-\tt y)\}}+ C_2\left( y-\tt y  +(y-\tt y)^{2-\aa}+ (y-\tt y)^{2-\bb} \right)\\
&\le -\lambda  |x-\tilde x| +2\lambda  (y-\tt y) +C_2\left(  y-\tt y  +(y-\tt y)^{2-\aa}+ (y-\tt y)^{2-\bb} \right)\\
&\le -\lambda  |x-\tilde x|+ C_3\left( y-\tt y +(y-\tt y)^{2-(\aa\vee \bb)} \right)\\
&\le-\lambda  |x-\tilde x|+ C_3\left(  y-\tt y +(y-\tt y )\I_{\{y-\tt y\ge1\}}  +(y-\tt y)^\theta\I_{\{0<y-\tt y\le1\}}\right)\\
&\le   -\lambda  |x-\tilde x|+2C_3(  y-\tt y  +(y-\tt y)^\theta),
\end{align*}
 where in the second inequality we used, due to $\aa,\bb>1$,    $(y-\tt y)^{1 -\aa}\wedge(y-\tt y)^{1 -\bb}\le1$ for $y-\tt y\ge1$  and
 $
 1\le (y-\tt y)^{-\aa}\vee (y-\tt y)^{-\bb}\le (y-\tt y)^{-(\aa\vee\bb)}$  for $0<y-\tt y\le 1$, and in the last two inequality we utilized $\theta\in(0,2-(\aa\vee\bb)].$ Therefore, to obtain the desired assertion \eqref{T1}, we need only to verify \eqref{K2}.

 From \eqref{WW3}, we have
\begin{equation}\label{T3}
\begin{split}
&(L^*F)(y-\tt y,|x-\tt x|)\\
&=-b\partial_1F (y-\tt y,|x-\tt x|)(y-\tt
y) -\lambda \partial_2F(y-\tt y,|x-\tt x|)|x-\tt x|\\
&\quad +(y -\tt y ) \int_0^\infty\bigg( F(y -\tt y+z,|x-\tt
x|)-F(y-\tt
y,|x-\tt x|)   \\
&\qquad\qquad\qquad\qquad\qquad\quad-\partial_1F(y-\tt y,|x-\tt x|)z\bigg)\,\nu_\bb(\d z)\\
&\quad +  (y -\tt y )\int_0^\infty \bigg(F(y -\tt y,|x+z-\tt
x|)-F(y-\tt
y,|x-\tt x|)   \\
&\qquad\qquad\qquad\qquad\qquad\quad-\partial_2F(y-\tt
y,|x-\tt x|)\ff{(x-\tt
x)z}{|x-\tt x|}\bigg)\,\nu_\aa(\d z)\\
&=:\Upsilon_1+\Upsilon_2+(y-\tt y) \Upsilon_3+(y-\tt y)\Upsilon_4.
\end{split}
\end{equation}
Thus, \eqref{K2} is available once we
prove that there is a constant $C_4>0$ so that
\begin{equation}\label{K3}
\begin{split}
\Upsilon_1+\Upsilon_2+(y-\tt y)\Upsilon_4
&\le  -\lambda |x-\tilde x|\I_{\{|x-\tt x|\ge2(y -\tt y)\}}+ C_4\left(  y-\tt y  +(y-\tt y)^{2-\aa}\right),
\end{split}
\end{equation}
and
\begin{equation}\label{K4}
 \Upsilon_3\le C_4(y-\tt y)^{1-\bb}.
\end{equation}

In what follows, we aim to prove \eqref{K3} and \eqref{K4}, respectively.
Let {\bf (i)}-{\bf (iii)} be the three cases listed in the proof of Proposition \ref{P:P-1}.
By a close inspection of argument for  Proposition \ref{P:P-1},
there is a constant $c_1>0$ such that
\begin{equation}\label{T4}\begin{split}
 \Upsilon_2+(y-\tt y)\Upsilon_4
&\le  -\lambda  |x-\tilde x| \I_{\{|x-\tt x|\ge 2(y-\tt y) \}}+ c_1\left( y-\tt y +(y-\tt y)^{2-\aa}\right).\end{split}\end{equation}
From \eqref{E6}, it is easy to see  that $\Upsilon_1\le bc_0(y-\tt y)$. As a result, \eqref{K3} follows immediately.
Next, we turn to the proof of \eqref{K4}.

\medskip

\noindent {\bf Proof of \eqref{K4} for the case (i)}. For this case, it follows that
\begin{align*}
\Upsilon_3=& \int_{\{|x-\tt x|\le
y-\tt y+z\}}\left(  y -\tt y+z   -|x-\tt x|  \right)\,\nu_\bb(\d z)\\
&+\int_{\{y-\tt y+z<|x-\tt x|<2(y-\tt
y+z)\}}\left(1-g\left(\ff{|x-\tt x|}{ y -\tt y+z
}\right)\right) \left( y -\tt y+z  - |x-\tt x| )\right)\,\nu_\bb(\d
z)\\
=&:\Upsilon_{31}+\Upsilon_{32}.
\end{align*}
On one hand, since $y-\tt y+z<|x-\tt x|$ and $g\in[0,1]$, $\Upsilon_{32}\le 0$.
On the other hand,  we have
$$
\Upsilon_{31}\le\int_{\{z\ge y-\tt y \}}  (y -\tt y+z)
\,\nu_\bb(\d z) \le c_2(y-\tt y)^{1-\bb}
$$
for some constant $c_2>0$,
where we used  $|x-\tt
x|\ge2(y-\tt y)$ in the first inequality. Hence, we arrive at
$$\Upsilon_3\le c_2(y-\tt y)^{1-\bb}.$$
Whence, we infer that \eqref{K4} holds true for the case {\bf(i)}.

\medskip

\noindent{\bf Proof of \eqref{K4} for the case (ii)}. Indeed, with regard to this case, \eqref{K4} is available by using
\begin{equation*}
\begin{split}\Upsilon_3&=\int_{\{|x-\tt x|<y-\tt y+z\}}\bigg(\left(1-g\left(\ff{|x-\tt x|}{ y -\tt y+z
}\right)\right)(y-\tt y+z) -(y-\tt y)-z\bigg)\,\nu_\bb(\d z)\\
&\le\int_{\{|x-\tt x|<y-\tt y+z\}}( y-\tt y+z -(y-\tt y)-z)\nu_\bb(\d z)=0,\end{split}\end{equation*}
where the inequality is due to $g\in[0,1]$.

\medskip

\noindent{\bf Proof of \eqref{K4} for the case (iii)}. As for
this case, $\Upsilon_3$  can be rewritten as
\begin{align*}
\Upsilon_3&=\int_{\{z\le (y-\tt y)/2\}}\bigg( F(y -\tt
y+z,|x-\tt x|)\\
&\quad~~~~~~~~~~~~~~~~~~~~~-F(y-\tt
y,|x-\tt x|)-\partial_1F(y-\tt y,|x-\tt x|)z\bigg)\,\nu_\bb(\d z)\\
&\quad+\int_{\{z\ge (y-\tt y)/2\}}\bigg( F(y -\tt y+z,|x-\tt
x|)-F_\theta(y-\tt
y,|x-\tt x|)\\
&~~~~~~~~~~~~~~~~~~~~~~~~~~-\partial_1F(y-\tt y,|x-\tt x|)z\bigg)\,\nu_\bb(\d z)\\
&=:\Upsilon_{31}+\Upsilon_{32}.
\end{align*}
By Taylor's expansion, there exists a constant $\xi\in
[y-\tt y,3(y-\tt y)/2]$ such that
\begin{equation}\label{K5}
\Upsilon_{31}=\ff{1}{2}\partial_{11}F (\xi,|x-\tt x|)\int_{\{z\le
(y-\tt y)/2\}}z^2\,\nu_\bb(\d z)\le c_3(y-\tt y)^{1-\bb}
\end{equation}
for some constant $c_3>0,$ where the inequality above follows from \eqref{E6}. Furthermore, in terms of  the definition of $F $ and \eqref{E6}, there is   a constant  $c_4>0$ such that
\begin{equation}\label{K6}
\begin{split}
\Upsilon_{32}&\le (1+c_0)\int_{\{z\ge (y-\tt y)/2\}}\left( y -\tt y  +|x-\tt x| +2z\right)\,\nu_\bb (\d z)\\
&\le
3(1+c_0)\int_{\{z\ge (y-\tt y)/2\}}\left( y -\tt y  +2z\right)\nu_\bb(\d z)
\le c_4(y-\tt y)^{1-\bb},
\end{split}
\end{equation}
 where  the second inequality is owing  to  $y-\tt y\le  |x-\tt x|\le 2(y-\tt y)$. Consequently, for the case {\bf (iii)}, \eqref{K4} follows  by combining \eqref{K5} with \eqref{K6}.
\end{proof}

\begin{proof}[Proof of Theorem $\ref{T:main1}$] Similar to Theorem \ref{T:main}, we indeed can claim that
there exist a unique probability measure $\mu$ on $\R_+\times\R$ such that for any $\theta\in (0, 2-(\aa\vee\bb)]$, there is $\eta:=\eta(\theta)>0$ so that for all $y\in \R_+$, $x\in \R$ and $t>0$
$$W_{\psi_\theta}  (P(t, (y,x),\cdot), \mu)\le C(y,x,\theta)\e^{-\eta t},$$ where
$$\psi_\theta(u,v):=u+ u^{\theta}+ v,\quad u,v\ge0$$ and $C(y,x,\theta)>0$ is independent of $t$.
Note that, as far as \eqref{WW2} is concerned,
one can  check that \eqref{e:mom} still holds. With \eqref{e:mom} and Proposition \ref{P:es2}
at hand, the proof of the assertion above can be complete by implementing  the same argument of Theorem \ref{T:main}, and so we herein omit the corresponding details.
\end{proof}

\subsection{General two-factor affine processes}
In this part, we emphasize  that the approaches applied to  Theorems
\ref{T:main} and \ref{T:main1} still work for two other general two-factor affine processes. In particular, we can show the exponential ergodicity for the following two kinds of two-factor affine processes  with respect to the $L^1$-Wasserstein distance $W_1$:

\smallskip
Two-factor affine process (I):
 $$\begin{cases}
\d Y_t=(a-b Y_t)\,\d t+  Y_{t}^{{1}/{2}} \,\d B^{(1)}_t + Y_{t-}^{{1}/{\beta}}\,\d Z^{(\bb)}_t,&\, t\ge0,\,Y_0\ge0,\\
\d X_t=(\kappa-\lambda X_t)\,\d t+ Y_{t}^{{1}/{2}} \,(\rho\,\d B^{(1)}_t+ \sqrt{ 1-
\rho^2}\,\d B^{(2)}_t) + Y_{t-}^{1/\aa}\,\d Z^{(\aa)}_t,&\,t\ge0,\,X_0\in \R,
\end{cases} $$ where $a\ge0$, $\kappa\in \R$, $b,\lambda>0$, $\rho\in [-1,1]$, $\bb,\aa\in(1,2)$, $(B_t^{(1)},B_t^{(2)})_{t\ge0}$ is a 2-dimensional standard Brownian motion, $(Z^{(\bb)}_t)_{t\ge0}$ (resp.\ $(Z^{(\aa)}_t)_{t\ge0}$)  is a specially positive $\bb$-stable  (resp. $\aa$-stable) L\'evy process. Moreover, the processes $(B_t^{(1)},B_t^{(2)})_{t\ge0}$, $(Z^{(\bb)}_t)_{t\ge0}$ and $(Z^{(\aa)}_t)_{t\ge0}$ are mutually independent.

\smallskip

 Two-factor affine process (II):
$$\begin{cases}
\d Y_t=(a-b Y_t)\,\d t+  Y_{t}^{{1}/{2}} \,\d B_t,&\quad t\ge0,\,Y_0\ge0,\\
\d X_t=(\kappa-\lambda X_t-\gamma Y_t)\,\d t+ Y_{t-}^{1/\aa}\,\d Z_t,&\quad t\ge0,\,X_0\in \R,
\end{cases}
$$where $a\ge 0$, $b,\lambda>0$, $\gamma \in
\R,$ $(B_t)_{t\ge0}$ is a standard Brownian motion, and $(Z_t)_{t\ge0}$ is an independent
spectrally  positive $\aa$-stable process with $\aa\in (1,2]$.

We have proven the exponential ergodicity for the SDEs \eqref{WW1} and \eqref{WW2} (with respect to the $L^1$-Wasserstein distance $W_1$). Combining both arguments together, we can show that  the following two-factor affine process
\begin{equation}\label{e:WW4}\begin{cases}
\d Y_t=(a-b Y_t)\,\d t+  Y_{t}^{{1}/{2}} \,\d B_t+ Y_{t-}^{{1}/{\beta}}\d Z^{(\bb)}_t,&\quad t\ge0,\,Y_0\ge0,\\
\d X_t=(\kappa-\lambda X_t)\,\d t+ Y_{t-}^{1/\aa}\,\d Z^{(\aa)}_t,&\quad t\ge0,\,X_0\in \R
\end{cases}\end{equation}
is exponentially ergodic.
The difference between the SDE \eqref{e:WW4} and type (I) above is  that in (I) there is an additional diffusion term driven by Brownian motions for the second equation. Then, one can apply the synchronous coupling to this additional term and follow the argument showing that  \eqref{e:WW4} is exponentially ergodic to  derive  the exponential ergodicity of  the type (I).

The difference between the SDE \eqref{WW1} and type (II) above is due to the drift term in the second equation for the process $(Y_t)_{t\ge0}$ of the associated affine process. Concerning the two-factor affine process (II), we  can  also show that $(Y_t,X_t)$  is  exponentially ergodic  with respect to the $L^1$-Wasserstein distance $W_1$ by following the argument of Theorem \ref{T:main} if  the counterpart of Proposition \ref{P:P-1} is still valid,   see Proposition \ref{p3} below for more details.

To proceed, we introduce some additional notation.
Below, let $F$ be defined as in \eqref{EE3} but with
$g\in C_b^2(\R_+)$ such that $g'\ge0$ and
$$
g(r)=
\begin{cases}
0,&0\le r<1,\\
(r-1)^{2+\delta},&1<r<{3}/{2},\\
1,&r\ge\kk_0:=2(1+ |\gamma|/\lambda)
\end{cases}$$
for some constant $\delta>0$. With $F$ above at hand, the function $V_{c,\theta}$ is defined exactly as in \eqref{v:function}.

\begin{proposition}\label{p3}
There exist constants $c,\eta>0$ such that for any $y>\tilde y\ge0$
 and $x,\tilde x\in \R$,
$$
(L^*V_{c,\theta})(y-\tilde y,|x-\tilde x|)\le -\eta V_{c,\theta} (y-\tilde y,|x-\tilde x|),
$$
where $L^*$ is the coupling operator of the two-factor affine process $($II$)$ given in \eqref{L2}
 with $\kk-\lambda x$ replaced by $\kk-\lambda x-\gamma y$.

\end{proposition}
\begin{proof}
Below, we shall fix $y>\tt y\ge0$ and $x,\tt x\in\R$.
By following the argument  of Proposition of \ref{P:P-1}, to end the proof of Proposition \ref{p3}, it is sufficient to show that
\begin{equation}\label{L1}
\begin{split}
(L^*F)
(y-\tt y,|x-\tt x|)&\le -\ff{\lambda}{2}  |x-\tt
x|\I_{\{|x-\tt x|\ge\kk_0(y-\tt y)  \}}\\
&\quad + C    (y-\tt y) +
C(\ss y+\ss{\tt
y})^2(y-\tt y)^{-1}\I_{\{0<y-\tt y<1\}}
\end{split}
\end{equation}
for some constant $C>0.$
 In what follows, we are going to verify \eqref{L1} in terms of the  three cases below:
\begin{enumerate}
\item[{\bf(i)}]   $|x-\tt x|> \kk_0(y-\tt y) $;

\item[{\bf(ii)}] $|x-\tt x|\le y-\tt y$;

\item[{\bf(iii)}]$y-\tt y<|x-\tt x|\le \kk_0(y-\tt y)$.
\end{enumerate}
For the  setting {\bf (i)}, we have  $F(y-\tt y,|x-\tt x|)= |x-\tt x|$. This, in addition to  \eqref{E5} with   $-\lambda|x-\tt x|-\gamma(x-\tt x)(y-\tt y)/|x-\tt x|$ instead of $-\lambda|x-\tt x|$ therein, yields
\begin{align*}
&(L^*F) (y-\tt y,|x-\tt x|)\\
& =-\lambda |x-\tt x|-\frac{\gamma (x-\tt x)(y-\tt y)}{|x-\tt x|}+(y-\tt y)(I_1+I_2+I_3)\\
&\le -\lambda  |x-\tt x|+|\gamma| (y-\tt y)+(y-\tt y)(I_1+I_2+I_3)\\
&\le -\frac{\lambda}{2}   |x-\tt
x|+(y-\tt y)(I_1+I_2+I_3),
\end{align*}
in which  $I_1,I_2,I_3$ were introduced in \eqref{L3} with the factor $2$ in the splitting intervals replaced by the number $\kk_0$, and in the last display  we used $$|\gamma|(y-\tt y)\le \frac{\lambda}{2}|x-\tt x|.$$
Thus, combining the estimates on $I_1,I_2,I_3$, we infer that \eqref{L1} holds true for the case {\bf (i)}. Observe that, as for the case  {\bf (ii)}, $L^*F$ shares the same expression with the counterpart for the setup {\bf (ii)}  in the proof of Proposition of \ref{P:P-1}. Subsequently, with some mild modifications of the associated details (where, in particular, replace the factor $2$ in the splitting intervals by $\kk_0$), we conclude that \eqref{L1} is still available for the case {\bf (ii)}. Note that Lemma \ref{L:lem1} is still valid with the number $2$ in \eqref{E6} replaced by $\kk_0 $ so that \eqref{E0} remains true for the new expression  $F.$ As a result, \eqref{L1} follows by carrying out the same argument to derive \eqref{PP} for the case {\bf (iii)}.
\end{proof}

  We mention that,  applying the similar idea, one can also prove the exponential ergodicity  with respect to the $L^1$-Wasserstein distance $W_1$ for type (II) if the first equation is replaced by
$$\d Y_t=(a-b Y_t)\,\d t+  Y_{t-}^{{1}/{\beta}}\,\d Z^{(\bb)}_t,\quad t\ge0,\,Y_0\ge0$$ or
$$\d Y_t=(a-b Y_t)\,\d t+  Y_{t}^{{1}/{2}} \,\d B_t+ Y_{t-}^{{1}/{\beta}}\,\d Z^{(\bb)}_t,\quad t\ge0,\,Y_0\ge0,$$ where $a\ge 0$, $b>0$, $(B_t)_{t\ge0}$ is a standard Brownian motion, and $(Z_t)_{t\ge0}$ is an independent
spectrally  positive $\aa$-stable process with $\aa\in (1,2]$.

\subsection{Beyond two-factor affine processes}\label{subsection4.3}
In this part, we will briefly mention that our method also works for the following models beyond two-factor affine processes.
Let $(Y_t,X_t)_{t\ge0}$ be a time-homogeneous Markov process on $\R_+\times \R$
such that
\begin{equation}\label{e:ext}
\begin{cases}
\d Y_t=b_1(Y_t)\,\d t+  Y_{t-}^{{1}/{\bb}} \,\d L_t,&\quad t\ge0,\,Y_0\ge0,
\\
\d X_t=b_2(X_t)\,\d t+ Y_{t-}^{1/\aa}\,\d Z_t,&\quad t\ge0,\, X_0\in \R,
\end{cases}
\end{equation} where $(L_t)_{t\ge0}$ is a spectrally positive $\bb$-stable process with $\bb\in (1,2]$, $(Z_t)_{t\ge0}$ is an independent
spectrally  positive $\aa$-stable process with $\aa\in (1,2]$, and  $b_1(x)$ (resp.\ $b_2(x)$) is continuous on $\R_+$ (resp.\ $\R$) so that there are constants $\lambda_i>0$ $(i=1,2)$ such that for $i=1,2$ and any $x>y$,
\begin{equation}\label{e:fff}b_i(x)-b_i(y)\le -\lambda_i (x-y).\end{equation} Note that, according to \cite[Theorem 5.6]{FL10}, there exists a unique strong solution to the first component $(Y_t)_{t\ge0}$ of the SDE \eqref{e:ext}. Once $(Y_t)_{t\ge0}$ is fixed, the unique strong solution to the second component $(X_t)_{t\ge0}$ is guaranteed by the monotone condition \eqref{e:fff} for the drift term $b_2(x)$. Therefore, the SDE \eqref{e:ext} has the unique strong solution $(Y_t,X_t)_{t\ge0}$. Furthermore, according to \cite[Corollary 2.3 and Remark 2.4]{LW}, the coupling of the first component $(Y_t)_{t\ge0}$ can be chosen to preserve the
order property. With these facts and \eqref{e:fff} again at hand, we can use
the ideas of proofs for Theorems \ref{T:main} and \ref{T:main1} to conclude that the process $(Y_t,X_t)_{t\ge0}$ is exponentially ergodic with respect to the $L^1$-Wasserstein distance $W_1$. Note that the contractive property like \eqref{Y6---} in the proof of Theorem \ref{T:main} is a consequence of \eqref{e:fff} for the drift term $b_2(x)$. As we mentioned before, since all
known approaches dealing
with the ergodicity of affine processes
depend on their
especially structural characterizations, they seem to be invalid in establishing the exponential ergodicity of the SDE \eqref{e:ext}.

\bigskip

\noindent \textbf{Acknowledgements.}
The research of Jianhai Bao is supported by the National Natural Science Foundation of China (Nos.\ 11771326, 11831014).
The research of Jian Wang is supported by the National Natural Science Foundation of China (No.\ 11831014),
the Program for Probability and Statistics:
Theory and Application (No.\ IRTL1704)
and the Program for Innovative Research Team in Science and Technology
in Fujian Province University (IRTSTFJ).

\end{document}